\documentclass[12pt,oneside]{article}
\usepackage{amsmath,amssymb,amsfonts,amsthm}
\usepackage{centernot}
\usepackage{stmaryrd}
\usepackage[pdftex]{graphicx}
\usepackage{lscape}
\newcommand{\T}{{\cal T}}

\newcommand{\set}[1]{\left\{#1\right\}}

\newcommand {\cp}{\mathfrak{X}(\pi (M))}

\def\Section#1{\vspace{30truept}\addtocounter{section}{1}\setcounter{thm}{0}
\setcounter{equation}{0}{\noindent\Large\bf
    \arabic{section}.~~#1}\par \vspace{12pt}}

\newtheorem{thm}{Theorem}[section]

\newtheorem{lem}[thm]{Lemma}
\newtheorem{prop}[thm]{Proposition}
\newtheorem{defn}[thm]{Definition}

\newtheorem{rem}[thm]{Remark}

\numberwithin{equation}{section}
\begin{document}
\title{{\textbf{Some Types of Recurrence in Finsler geometry}}} %\footnote{arXiv: 1607.07468 [math.DG]}}
\author{{\bf A. Soleiman$^{1}$  and  Nabil L. Youssef\,$^{\,2}$}}
\date{}

\maketitle                     % Produces the title.
\vspace{-1.15cm}
\begin{center}
{$^{1}$Department of Mathematics, Faculty of Science,\\ Benha
University, Benha, Egypt\\amr.hassan@fsci.bu.edu.eg, amrsoleiman@yahoo.com}
\end{center}
\begin{center}
{$^{2}$Department of Mathematics, Faculty of Science,\\ Cairo
University, Giza, Egypt\\nlyoussef@sci.cu.edu.eg, nlyoussef2003@yahoo.fr}
\end{center}
%\vspace{-0.6cm}

\vspace{0.2cm}
\maketitle

\vspace{0.2cm}
\hfill \emph{Dedicated to the memory of Waleed A. Elsayed}
\vspace{0.6cm}

%%%%%%%%%%%%%%%%%%%%%%%%%%%%%%%%%%%%%%%%%%%%%%%%%%%%%%%%%%%%%%% Abstract %%%%%%%%%%%%%%%%%%%%%%%%%%%%%%%%%%%%%%%%%%%%%%%%%%%%%%%%%%%

\noindent{\bf Abstract.}
The pullback  approach to global Finsler geometry is adopted.
Three classes of recurrence in Finsler geometry  are introduced and investigated: simple recurrence, Ricci recurrence and concircular recurrence.
Each of these classes consists of four types of recurrence.
The interrelationships between the different types of recurrence are studied.
The generalized concircular recurrence, as a new concept, is singled out.


\bigskip

%%%%%%%%%%%%%%%%%%%%%%%%%%%%%%%%%%%%%%%%%%%%%%%%%%% Introduction %%%%%%%%%%%%%%%%%%%%%%%%%%%%%%%%%%%%%%%%%%%%%%%%%%%%%%%%

\begin{center}
\large{\bf Introduction}.
\end{center}
Many types of recurrent Riemannian manifolds have been studied by many authors (e.g., \cite{R3, R4, R2, R5, R1}).
 On the other hand, some types of recurrent Finsler spaces have been also studied (e.g., \cite{F3, F2, F1}).
\par
In this paper, we gather all known types of Finsler recurrence (related to Cartan connection), besides some new ones, in a single general setting.
We study \emph{intrinsically} three classes of recurrence: simple recurrence, Ricci recurrence and concircular recurrence.
Each of these classes consists of four types of recurrence.
The interrelationships between the different types of recurrence are investigated. A special emphasis is focused on the new concept of generalized concircular recurrence.
At the end of the paper we provide a concise diagram presenting the relationships among the different types of Finsler recurrences treated.
\newpage
%%%%%%%%%%%%%%%%%%%%%%%%%%%%%%%%%%%%%% SECTION 1. Notation and Preliminaries %%%%%%%%%%%%%%%%%%%%%%%%%%%%%%%%%%%%%%

\Section{Notation and Preliminaries}

In this section, we give a brief account of the basic concepts
 of the pullback approach to intrinsic Finsler geometry necessary for this work. For more
 details, we refer to \cite{r58, r86, r94, r96}. We
 shall use the notations of \cite{r86}.

 In what follows, we denote by $\pi: \T M\longrightarrow M$ the subbundle of nonzero vectors
tangent to $M$, $\mathfrak{F}(TM)$ the algebra of $C^\infty$ functions on $TM$, $\cp$ the $\mathfrak{F}(TM)$-module of differentiable sections of the pullback bundle $\pi^{-1}(T M)$.
The elements of $\mathfrak{X}(\pi (M))$ will be called $\pi$-vector
fields and will be denoted by barred letters $\overline{X} $. The
tensor fields on $\pi^{-1}(TM)$ will be called $\pi$-tensor fields.
The fundamental $\pi$-vector field is the $\pi$-vector field
$\overline{\eta}$ defined by $\overline{\eta}(u)=(u,u)$ for all
$u\in TM$.
\par
We have the following short exact sequence of vector bundles
$$0\longrightarrow
 \pi^{-1}(TM)\stackrel{\gamma}\longrightarrow T(\T M)\stackrel{\rho}\longrightarrow
\pi^{-1}(TM)\longrightarrow 0 ,\vspace{-0.1cm}$$ with the well known
definitions of  the bundle morphisms $\rho$ and $\gamma$. The vector
space $V_u (\T M)= \{ X \in T_u (\T M) : d\pi(X)=0 \}$  is the vertical space to $M$ at $u$.
\par
Let $D$ be  a linear connection on the pullback bundle $\pi^{-1}(TM)$.
 We associate with $D$ the map \vspace{-0.1cm} $K:T \T M\longrightarrow
\pi^{-1}(TM):X\longmapsto D_X \overline{\eta} ,$ called the
connection map of $D$.  The vector space $H_u (\T M)= \{ X \in T_u
(\T M) : K(X)=0 \}$ is called the horizontal space to $M$ at $u$ .
   The connection $D$ is said to be regular if
$$ T_u (\T M)=V_u (\T M)\oplus H_u (\T M) \,\,\,  \forall \, u\in \T M.$$

If $M$ is endowed with a regular connection, then the vector bundle
   maps $
 \gamma,\, \rho |_{H(\T M)}$ and $K |_{V(\T M)}$
 are vector bundle isomorphisms. The map
 $\beta:=(\rho |_{H(\T M)})^{-1}$
 will be called the horizontal map of the connection
$D$.
\par
 The horizontal ((h)h-) and
mixed ((h)hv-) torsion tensors of $D$, denoted by $Q $ and $ T $
respectively, are defined by \vspace{-0.2cm}
$$Q (\overline{X},\overline{Y})=\textbf{T}(\beta \overline{X}\beta \overline{Y}),
\, \,\,\, T(\overline{X},\overline{Y})=\textbf{T}(\gamma
\overline{X},\beta \overline{Y}) \quad \forall \,
\overline{X},\overline{Y}\in\mathfrak{X} (\pi (M)),\vspace{-0.2cm}$$
where $\textbf{T}$ is the (classical) torsion tensor field
associated with $D$.
\par
The horizontal (h-), mixed (hv-) and vertical (v-) curvature tensors
of $D$, denoted by $R$, $P$ and $S$
respectively, are defined by
$$R(\overline{X},\overline{Y})\overline{Z}=\textbf{K}(\beta
\overline{X}\beta \overline{Y})\overline{Z},\quad
 {P}(\overline{X},\overline{Y})\overline{Z}=\textbf{K}(\beta
\overline{X},\gamma \overline{Y})\overline{Z},\quad
 {S}(\overline{X},\overline{Y})\overline{Z}=\textbf{K}(\gamma
\overline{X},\gamma \overline{Y})\overline{Z}, $$
 where $\textbf{K}$
is the (classical) curvature tensor field associated with $D$.
\par
The contracted curvature tensors of $D$, denoted by $\widehat{{R}}$, $\widehat{ {P}}$ and $\widehat{ {S}}$ (known
also as the (v)h-, (v)hv- and (v)v-torsion tensors respectively), are defined by
$$\widehat{ {R}}(\overline{X},\overline{Y})={ {R}}(\overline{X},\overline{Y})\overline{\eta},\quad
\widehat{ {P}}(\overline{X},\overline{Y})={
{P}}(\overline{X},\overline{Y})\overline{\eta},\quad \widehat{
{S}}(\overline{X},\overline{Y})={
{S}}(\overline{X},\overline{Y})\overline{\eta}.$$
%\par
%The following result is of extreme importance. \vspace{-0.1cm}
\begin{thm} {\em\cite{r94}} \label{th.1} Let $(M,L)$ be a Finsler
manifold and  $g$ the Finsler metric defined by $L$. There exists a
unique regular connection $\nabla$ on $\pi^{-1}(TM)$ such
that\vspace{-0.2cm}
\begin{description}
  \item[(a)]  $\nabla$ is  metric\,{\em:} $\nabla g=0$,

  \item[(b)] The (h)h-torsion of $\nabla$ vanishes\,{\em:} $Q=0
  $,
  \item[(c)] The (h)hv-torsion $T$ of $\nabla$\, satisfies\,\emph{:}
   $g(T(\overline{X},\overline{Y}), \overline{Z})=g(T(\overline{X},\overline{Z}),\overline{Y})$.
\end{description}
\par
 Such a connection is called the Cartan
connection associated with  the Finsler manifold $(M,L)$.
\end{thm}

The only linear connection we deal with in this paper is the Cartan connection.

%%%%%%%%%%%%%%%%%%%%%%%%%%%%%%%%%%%%%%%%%%%%%%%%%%%%%%%%%% SECTION 2. %%%%%%%%%%%%%%%%%%%%%%%%%%%%%%%%%%%%%%%%%%%%%%%%%%%%%%%%

\Section{Three types of Finsler Recurrence}

In this section, we introduce three classes of recurrent Finsler spaces which
will be the object of our investigation in the next sections. These notions are defined in Riemannian geometry \cite{R3, R4, R2, R5, R1}. We extend them to the Finslerian  case.

\smallskip
For a Finsler manifold  $(M,L)$,  we set the following notations:
\vspace{-8pt}
\begin{eqnarray*}
\stackrel{h}\nabla&:&\text{the $h$-covariant derivatives associated
with Cartan connection},\\
\text{Ric} &:& \text{the  horizontal  Ricci tensor of Cartan connection},\\
 r &:& \text{the horizontal scalar curvature of Cartan connection},\\
G(\overline{X},\overline{Y})\overline{Z} &:=& g(\overline{X},\overline{Z})
\overline{Y}-g(\overline{Y},\overline{Z})\overline{X},\\
C &:=& {R}-\frac{r}{n(n-1)}\, {G}: \text{ the concircular curvature tensor},\\
\textbf{R}(\overline{X},\overline{Y},\overline{Z},\overline{W}) &:=& g(R(X,Y)Z,W),\\
\textbf{G}(\overline{X},\overline{Y},\overline{Z},\overline{W}) &:=& g(G(X,Y)Z,W),\\
\textbf{C}(\overline{X},\overline{Y},\overline{Z},\overline{W}) &:=& g(C(X,Y)Z,W).
\end{eqnarray*}

A Finsler manifold is said to be horizontally integrable if its
horizonal distribution is completely integrable (or, equivalently, $\widehat{R}=0)$.

\smallskip

\begin{defn}\label{def.1a} Let $(M,L)$ be a Finsler manifold  of dimension $n\geq3$ with non-zero $h$-curvature tensor
 ${R}$. Then, $(M,L)$ is said to be:
 \begin{description}
   \item[(a)] recurrent Finsler manifold $(F_{n})$ if \,\,$\stackrel{h}{\nabla} {R}= A\otimes {R}$,
   \item[(b)] 2-recurrent Finsler manifold $(2F_{n})$ if \,\,$\stackrel{h}{\nabla}\stackrel{h}{\nabla} {R}= \alpha\otimes {R}$,
   \item[(c)] generalized recurrent Finsler manifold $(GF_{n})$ if \,\,$\stackrel{h}{\nabla} {R}= A\otimes {R}+B\otimes {G}$,
   \item[(d)] generalized 2-recurrent Finsler manifold $(G(2F_{n}))$ if \,\,$\stackrel{h}{\nabla}\stackrel{h}{\nabla} {R}= \alpha\otimes {R}+\mu \otimes {G}$,
 \end{description}
  where $A$ and $B$ (resp. $\alpha$ and $\mu$ ) are non-zero scalar 1-forms (resp. 2-forms) on $TM$, and
 positively homogenous of degree zero in $y$, called the recurrence forms.
\par In particular, if \,$\stackrel{h}{\nabla} {R}=0$, then $(M,L)$ is
called symmetric.
\end{defn}

\begin{defn}\label{def.2a} Let $(M,L)$ be a Finsler manifold  of dimension $n\geq3$ with non-zero horizontal Ricci tensor
 $\emph{\text{Ric}}$. Then, $(M,L)$ is said to be:
 \begin{description}
   \item[(a)] Ricci recurrent Finsler manifold $(RF_{n})$ if \,\,$\stackrel{h}{\nabla} \emph{\text{Ric}}= A\otimes \emph{\text{Ric}}$,
   \item[(b)] 2-Ricci recurrent Finsler manifold $(2RF_{n})$ if \,\,$\stackrel{h}{\nabla}\stackrel{h}{\nabla} \emph{\text{Ric}}= \alpha\otimes \emph{\text{Ric}}$,
   \item[(c)] generalized Ricci recurrent Finsler manifold $(GRF_{n})$ if \,\,$\stackrel{h}{\nabla} \emph{\text{Ric}}= A\otimes \emph{\text{Ric}}+B\otimes {g}$,
   \item[(d)] generalized 2-Ricci recurrent Finsler manifold $(G(2RF_{n}))$ if $$\stackrel{h}{\nabla}\stackrel{h}{\nabla} \emph{\text{Ric}}= \alpha\otimes \emph{\text{Ric}}+\mu \otimes {g},$$
 \end{description}
  where $A$ and $B$ (resp. $\alpha$ and $\mu$ ) are as given in Definition \ref{def.1a}.

In particular, if \,$\stackrel{h}{\nabla} \emph{\text{Ric}}=0$, then $(M,L)$ is
called Ricci symmetric.
\end{defn}

\begin{defn}\label{def.3a} Let $(M,L)$ be a Finsler manifold  of dimension $n\geq3$ with non-zero concircular curvature tensor $C$. Then, $(M,L)$ is said to be:
 \begin{description}
   \item[(a)] concircularly recurrent Finsler manifold $(CF_{n})$ if \,\,$\stackrel{h}{\nabla} {C}= A\otimes {C},$
   \item[(b)] 2-concircularly recurrent Finsler manifold $(2CF_{n})$ if \,\,$\stackrel{h}{\nabla}\stackrel{h}{\nabla} {C}= \alpha\otimes {C},$
   \item[(c)] generalized concircularly recurrent Finsler manifold $(GCF_{n})$ if
   $$\stackrel{h}{\nabla} {C}= A\otimes {C}+B\otimes {G},$$
   \item[(d)] generalized 2-concircularly recurrent Finsler manifold $(G(2CF_{n}))$ if
   $$\stackrel{h}{\nabla}\stackrel{h}{\nabla} {C}= \alpha\otimes {C}+\mu \otimes {G},$$
 \end{description}
  where $A$ and $B$ (resp. $\alpha$ and $\mu$ ) are as given in Definition \ref{def.1a}.
\par
In particular, if \,$\stackrel{h}{\nabla} {C}=0$, then $(M,L)$ is
called concircularly symmetric.
\end{defn}

We quote the following two Lemmas from \cite{F1}; they are very useful in the sequel.
\begin{lem}\label{lem.1}  For a horizontally integrable Finsler manifold, we have:
\begin{description}
  \item[(a)]$\mathfrak{S}_{\overline{X},\overline{Y},\overline{Z}}\, \{{R}(\overline{X},\overline{Y})\overline{Z}\}=0.$
  \footnote{$\mathfrak{S}_{\overline{X},\overline{Y},\overline{Z}}$
denotes the cyclic sum over ${\overline{X},\overline{Y},\overline{Z}}$.}
    \item[(b)]$\textbf{R}(\overline{X},\overline{Y},\overline{Z},\overline{W})=\textbf{R}(\overline{Z},\overline{W},\overline{X},\overline{Y}).$
  \item[(c)]  $\mathfrak{S}_{\overline{X},\overline{Y},\overline{Z}} \,\{{(\stackrel{h}{\nabla}R)}(\overline{X},\overline{Y},\overline{Z},\overline{W})\}=0.$
  \item[(d)] The horizontal Ricci tensors $Ric$ is symmetric.
  \item[(e)] $\mathfrak{S}_{\overline{U},\overline{V};\,\,\overline{W},\overline{X};\,\,\overline{Y},\overline{Z}}\set{({R}
(\overline{U},\overline{V})\textbf{R})(\overline{W},\overline{X},\overline{Y},\overline{Z})}=0.$
\footnote{$\mathfrak{S}_{\overline{U},\overline{V};\,\,\overline{W},\overline{X};\,\, \overline{Y},\overline{Z}}$
denotes the cyclic sum over the three pairs of arguments
$\overline{U},\overline{V};\,\, \overline{W},\overline{X};\,\,\overline{Y},\overline{Z}$.}
\item[(f)] $\mathfrak{S}_{\overline{U},\overline{V};\,\,\overline{W},\overline{X};\,\,\overline{Y},\overline{Z}}\set{({R}
(\overline{U},\overline{V})\textbf{C})(\overline{W},\overline{X},\overline{Y},\overline{Z})}=0$.
  \item[(g)]$ (\stackrel{h}{\nabla}\stackrel{h}{\nabla}\omega)(\overline{Y},\overline{X},\overline{Z})-
(\stackrel{h}{\nabla}\stackrel{h}{\nabla}\omega)(\overline{X},\overline{Y},\overline{Z})=
( {R}(\overline{X},\overline{Y})\omega)(\overline{Z})$;
$\omega$ is a  $\pi$(1)-form.
\end{description}
\end{lem}

\begin{lem}\label{lem.2} Let $(M,L)$ be a horizontally integrable Finsler manifold and let  $\omega$ be a $\pi$(2)-form. If any one of the following relations holds
\begin{eqnarray*}
  \mathfrak{S}_{\overline{U},\overline{V};\,\overline{W},\overline{X};\,\overline{Y},\overline{Z}}\set{\omega(\overline{U},\overline{V})\textbf{R}(\overline{W},\overline{X},\overline{Y},\overline{Z})
  }&=&0,\\
\mathfrak{S}_{\overline{U},\overline{V};\,\overline{W},\overline{X};\,\overline{Y},\overline{Z}}\set{\omega(\overline{U},\overline{V})\textbf{C}(\overline{W},\overline{X},\overline{Y},\overline{Z})
  }&=&0,\\
  \mathfrak{S}_{\overline{U},\overline{V};\,\overline{W},\overline{X};\,\overline{Y},\overline{Z}}\set{\omega(\overline{U},\overline{V})\textbf{G}(\overline{W},\overline{X},\overline{Y},\overline{Z})
  }&=&0,
\end{eqnarray*}
then $\omega$ vanishes identically.
\end{lem}

%%%%%%%%%%%%%%%%%%%%%%%%%%%%%%%%%%%%%%%%%%%%%%%%%%%%%%%%%%%%%%% Section 3 %%%%%%%%%%%%%%%%%%%%%%%%%%%%%%%%%%%%%%%%%%%%%%%%%%%%%%%%%%%%%%%%%%%%%%%

\Section{ Recurrenc (2-recurrence) }

\begin{prop} \label{prop.1} Let $(M,L)$ be a horizontally integrable Finsler manifold of dimension $n\geq3$.
If $(M,L)$ is recurrent (resp. 2-recurrent) with
recurrence form $A$ ( resp. $\alpha$), then we have:
  \begin{description}
  \item[(a)] $\mathfrak{S}_{\overline{X},\overline{Y},\overline{Z}} \,\{{(A\otimes R)}(\overline{X},\overline{Y},\overline{Z},\overline{W})\}=0$
  \item[(b)]  $\stackrel{h}{\nabla} A$ \emph{(}resp. $\alpha$\emph{)} is symmetric,
  \item[(c)]  $R(\overline{X},\overline{Y}){\bf R}=0$.
        \end{description}
  \end{prop}

\begin{proof}
The proof follows from  Definition \ref{def.1a} together with Lemmas \ref{lem.1} and \ref{lem.2}.
\end{proof}

\begin{thm} \label{thm.1} If $(M,L)$ is a recurrent Finsler manifold of dimension $n\geq3$ with
recurrence form $A$, then
\begin{description}
  \item[(a)] $(M,L)$ is $RF_{n} $.
  \item[(b)] $(M,L)$ is $ CF_{n} $ provided that $r\neq0$.
  \item[(c)] $(M,L)$ is $ 2F_{n} $ \emph{(} resp. $2RF_{n}$\emph{)} provided that $\stackrel{h}{\nabla}A+A\otimes A\neq0$.
   \end{description}
\end{thm}

\begin{proof} ~\par

\vspace{4pt}
 \noindent\textbf{(a)} is clear from the definitions of recurrence and Ricci recurrence.

 \vspace{4pt}

 \noindent\textbf{(b)}  As $(M,L)$ is a recurrent Finsler manifold,
 $ \stackrel{h}{\nabla} {R}= A\otimes {R}$,    with $A\neq 0 $.
Hence,
 $\stackrel{h}{\nabla} {r}= rA $,  with ${r\neq0}$ by assumption. Consequently,
     \begin{eqnarray}
        \stackrel{h}{\nabla} {C}&=& \stackrel{h}{\nabla} \set{{R}-\frac{r}{n(n-1)}\, {G}}
         =\,\, \stackrel{h}{\nabla}{R}-\frac{\stackrel{h}{\nabla}r}{n(n-1)}\otimes G,  {\quad\text{since} \stackrel{h}{\nabla}{G}=0}   \nonumber \\
         &{=}& A\otimes {R}-\frac{rA}{n(n-1)}\otimes  G
        = A\otimes C  \nonumber .
      \end{eqnarray}

\vspace{4pt}
 \noindent\textbf{(c)}  Using $\stackrel{h}{\nabla} {R}= A\otimes {R}$, we have
    \begin{equation*}\label{e1}
  \stackrel{h}{\nabla} \stackrel{h}{\nabla} {R}= \stackrel{h}{\nabla}A\otimes {R}+A\otimes \stackrel{h}{\nabla}{R}
        =(\stackrel{h}{\nabla}A +A\otimes A)\otimes {R}  = \alpha \otimes {R},
        \end{equation*}
   where $\alpha:=\stackrel{h}{\nabla}A  +A\otimes A$.  Hence, $(M,L)$ is  $2F_{n}$ provided that $\alpha\neq0$.

   Similarly, one can show that $(M,L)$ is  $2RF_{n}$.
 \end{proof}

\begin{rem}
One can easily show that, the sufficient condition for a  Ricci recurrent Finsler manifold of dimension $n\geq3$  with
recurrence form $A$ to be a 2-Ricci recurrent Finsler manifold is that $\stackrel{h}{\nabla}A  +A\otimes A \neq0 $.
\end{rem}

%%%%%%%%%%%%%%%%%%%%%%%%%%%%%%%%%%%%%%%%%%%%%%%%%%%%%%%%%%%% Section 4 %%%%%%%%%%%%%%%%%%%%%%%%%%%%%%%%%%%%%%%%%%%%%%%%%%%%%%%%%%%%%%%

\Section{Concircular recurrence (2-recurrence)}

\begin{prop} \label{prop.2} Let $(M,L)$ be a horizontally integrable concircularly  recurrent (resp. 2-concircularly recurrent) Finsler manifold of dimension $n\geq3$ with recurrence form $A$ (resp. $\alpha$), then we have:
\begin{description}
  \item[(a)]  $\stackrel{h}{\nabla} A$ \emph{(}resp. $\alpha$\emph{)} is symmetric,
  \item[(b)] $R(\overline{X},\overline{Y}){\bf C}=0$,
  \item[(c)] $R(\overline{X},\overline{Y}){\bf R}=0$.
   \end{description}
   \end{prop}

\begin{proof}
The proof follows from  Definition \ref{def.3a} together with Lemmas \ref{lem.1} and  \ref{lem.2}, after some calculations.
\end{proof}

\begin{thm}\label{thm.2} If $(M,L)$ is a  concircularly  recurrent Finsler manifold of dimension $n\geq3$ with
recurrence form $A$, then
\begin{description}
  \item[(a)]  $(M,L)$ is $ 2CF_{n} $ provided that $\stackrel{h}{\nabla}A+A\otimes A\neq0$.
  \item[(b)]  $(M,L)$ is $ GF_{n} $ \emph{(}resp. $ GRF_{n}$\emph{)} provided that $\stackrel{h}{\nabla}r-rA\neq0$.
  \item[(c)]  $(M,L)$ is $ F_{n} $ provided that $\widehat{R}=0$.
 \end{description}
\end{thm}

\begin{proof} ~\par

\vspace{4pt}
 \noindent\textbf{(a)}  Let $(M,L)$ be concircularly recurrent, then $\stackrel{h}{\nabla} {C}= A\otimes {C}$, {with} $ A\neq 0 $.
 Consequently,
      \begin{equation*}
             \stackrel{h}{\nabla} \stackrel{h}{\nabla} {C}= \stackrel{h}{\nabla}A\otimes {C}+A\otimes \stackrel{h}{\nabla}{C}
              = (\stackrel{h}{\nabla}A +A\otimes A)\otimes {C}
        = \alpha \otimes {C},
             \end{equation*}
where $\alpha:=\stackrel{h}{\nabla}A+A\otimes A$.  Hence, $(M,L)$ is a $2CF_{n}$ if $\alpha\neq0$.

 \vspace{4pt}

 \noindent\textbf{(b)} As $\stackrel{h}{\nabla} {C}= A\otimes {C}$, ${C}= {R}-\frac{r}{n(n-1)}\, {G}$  and $\stackrel{h}{\nabla} {G}=0$,
 we get
      \begin{eqnarray}
       \stackrel{h}{\nabla}{R}&{=}&A\otimes {R}+\frac{1}{n(n-1)}\{\stackrel{h}{\nabla}r-rA\}\otimes {G}   \label{eq.3}\\
        \stackrel{h}{\nabla}{R}&{=}& A\otimes {R}+B\otimes {G}  \nonumber ,
      \end{eqnarray}
where  $B=\frac{1}{n(n-1)}(\stackrel{h}{\nabla}{r}-rA)$. Since $A\neq0$, then, $(M,L)$ is $GF_{n}$ if $B\neq0$.

\par

Now, taking the trace of both sides of (\ref{eq.3}), one gets
$$\stackrel{h}{\nabla}{Ric}= A\otimes {Ric}+B_{1}\otimes g, $$
 where $B_{1}=\frac{1}{n}\{\stackrel{h}{\nabla}r-rA\}$. Hence, $(M,L)$ is $GRF_{n}$ if $B_{1}\neq0$.

\vspace{4pt}

 \noindent\textbf{(c)} Follows from Theorem C of \cite{F1}.
\end{proof}

%%%%%%%%%%%%%%%%%%%%%%%%%%%%%%%%%%%%%%%%%%%%%%%%%%%%%%%%%%%%% Section 5 %%%%%%%%%%%%%%%%%%%%%%%%%%%%%%%%%%%%%%%%%%%%%%%%%%%%%%%%%%%%%

\Section{Generalized recurrence (2-recurrence)}

\begin{thm}\label{thm.3} If $(M,L)$ is a generalized recurrent Finsler manifold of dimension $n\geq3$ with
recurrence forms $A $ and $B$, then
\begin{description}
\item[(a)]  $(M,L)$ is $G(2F_{n}) $ provided that $\stackrel{h}{\nabla}A +A\otimes A\neq0$ and $\stackrel{h}{\nabla}B +A\otimes B\neq0$.
\item[(b)]  $(M,L)$ is $CF_{n}$ provided that $r\neq0$.
\item[(c)]  $(M,L)$ is $2CF_{n} $ provided that $\stackrel{h}{\nabla}A +A\otimes A\neq0$ and $r\neq0$.
\item[(d)]  $(M,L)$ is $GRF_{n} $
\item[(e)]  $(M,L)$ is $ F_{n} $ provided that $\widehat{R}=0 $ and $r\neq0$.
 \end{description}
\end{thm}

\begin{proof} ~\par

\vspace{4pt}
 \noindent\textbf{(a)}  Let $(M,L)$ be a generalized recurrent Finsler manifold, then
  $ \stackrel{h}{\nabla} {R}= A\otimes {R}+B\otimes {G}$,    with $A\neq 0 \neq B $.
   Consequently,
      \begin{eqnarray}
        \stackrel{h}{\nabla} \stackrel{h}{\nabla} {R}&{=}& (\stackrel{h}{\nabla}A\otimes {R}+A\otimes \stackrel{h}{\nabla}{R})
       +\stackrel{h}{\nabla}B\otimes {G},  \,\, \text{since} \stackrel{h}{\nabla}G=0 \nonumber \\
        &{=}& (\stackrel{h}{\nabla}A\otimes {R}
         +A\otimes A\otimes {R}+A\otimes B\otimes {G})
       +\stackrel{h}{\nabla}B\otimes {G}   \nonumber \\
       &{=}& (\stackrel{h}{\nabla}A +A\otimes A)\otimes {R}+(A\otimes B +\stackrel{h}{\nabla}B)\otimes {G}  \label{eq.aa} \\
          &{=}& \alpha\otimes {R}+\mu\otimes {G}   \nonumber ,
      \end{eqnarray}
where $\alpha:=\stackrel{h}{\nabla}A +A\otimes A$ and $\mu:=\stackrel{h}{\nabla}B+A\otimes B$ are non-zero scalar 2-forms.

\vspace{4pt}
 \noindent\textbf{(b)} By double contraction of \,$ \stackrel{h}{\nabla} {R}= A\otimes {R}+B\otimes {G}$,\, we get
 \begin{eqnarray}\label{eq.con}
      \stackrel{h}{\nabla} {r}&=& rA + n(n-1)B .
      \end{eqnarray}
 Consequently,
      \begin{eqnarray}
        \stackrel{h}{\nabla} {C}&=& \stackrel{h}{\nabla} \set{{R}-\frac{r}{n(n-1)}\, {G}}=\,\,\stackrel{h}{\nabla}{R}-\frac{\stackrel{h}{\nabla}r}{n(n-1)}\otimes G  \nonumber \\
         &\stackrel{(\ref{eq.con})}{=}& A\otimes {R}-\frac{rA}{n(n-1)}\otimes  G  = A\otimes C  \label{eq.con1} .
      \end{eqnarray}

 \vspace{4pt}

\noindent\textbf{(c)} follows from \textbf{(b)},\,\, \textbf{(d)} is trivial,\,\,  \textbf{(e)} follows from \textbf{(b)} and Theorem \ref{thm.2}.
\end{proof}

\begin{rem} One can easily show that, the sufficient condition for a generalized Ricci recurrent Finsler manifold of dimension $n\geq3$,  with
recurrence forms $A$ and $B$, to be a generalized 2-Ricci recurrent Finsler manifold is that $\stackrel{h}{\nabla}A +A\otimes A\neq0$ and $\stackrel{h}{\nabla}B +A\otimes B\neq0$.
\end{rem}

\begin{prop} \label{prop.3} Let $(M,L)$ be a horizontally integrable generalized recurrent Finsler manifold of dimension $n\geq3$ with
recurrence forms $A, B$ and scalar curvature $r$, then we have:
  \begin{description}
  \item[(a)] $\mathfrak{S}_{\overline{X},\overline{Y},\overline{Z}} \,\{{(A\otimes R+B\otimes G)}(\overline{X},\overline{Y},\overline{Z},\overline{W})\}=0$.
  \item[(b)] $\overline{d}A$ \,and \,\,$\overline{d}B+A\wedge B$ vanish identically.
  \item[(c)] $R(\overline{X},\overline{Y}){\bf R}=0$,\\
   where $\overline{d}A(\overline{X},\overline{Y}):=(\stackrel{h}{\nabla}{A})(\overline{X},\overline{Y})
        -(\stackrel{h}{\nabla}{A})(\overline{Y},\overline{X})$.
\end{description}
     \end{prop}

\begin{proof} The proof of {(\bf a)} is easy.\\
Now, we prove {(\bf b)}.
 Let $(M,L)$ be horizontally integrable and generalized recurrent with
recurrence forms $A$ and $B$. Then, by Theorem \ref{thm.3}, $(M,L)$ is concirculary recurrent, i.e., $ \stackrel{h}{\nabla}C=A\otimes C$, by (\ref{eq.con1}). Consequently, $\stackrel{h}{\nabla}\textbf{C}=A\otimes\textbf{ C}$. Hence,

\begin{equation*}
  \stackrel{h}{\nabla}\stackrel{h}{\nabla}\textbf{C}= (\stackrel{h}{\nabla}A)\otimes \textbf{C}
  +A\otimes \stackrel{h}{\nabla}\textbf{C}
   =(\stackrel{h}{\nabla}A+A\otimes A)\otimes \textbf{C}.
\end{equation*}
From which, taking into account Lemma \ref{lem.1}{\bf(g)}, we obtain
\begin{eqnarray*}
   {R}(\overline{U},\overline{V})\textbf{C}&=&-({\overline{d}}A)(\overline{U},\overline{V})\textbf{C}.\label{eq.16a}
\end{eqnarray*}
Hence, using Lemma
\ref{lem.1}(\textbf{f}), it follows that
\begin{equation*}
   \mathfrak{S}_{\overline{U},\overline{V};\,\overline{W},\overline{X};\,\overline{Y},\overline{Z}}
   \set{{\overline{d}}A(\overline{U},\overline{V})\textbf{C}(\overline{W},\overline{X},\overline{Y},\overline{Z})
  }=0.
\end{equation*}
From which, together with Lemma \ref{lem.2}, we conclude that
\begin{equation}\label{da}
   {\overline{d}}A=0
\end{equation}

On the other hand, from (\ref{eq.aa}), we obtain
\begin{eqnarray*}
    \stackrel{h}{\nabla}\stackrel{h}{\nabla}\textbf{R}&=& (\stackrel{h}{\nabla}A+A\otimes A)\otimes \textbf{R}
  +(A\otimes B+ \stackrel{h}{\nabla}B)\otimes\textbf{G}.
\end{eqnarray*}
From which, taking into account (\ref{da}) and Lemma \ref{lem.1}, we get
\begin{eqnarray}
  {R}(\overline{U},\overline{V})\textbf{R}&=&
     -(\overline{d}B+A\wedge B)(\overline{U},\overline{V})\textbf{G}.\label{eq.16}
\end{eqnarray}
Hence, from  Lemma \ref{lem.1}(\textbf{f}), we obtain
\begin{eqnarray*}
   \mathfrak{S}_{\overline{U},\overline{V};\,\overline{W},\overline{X};\,\overline{Y},\overline{Z}}
   \set{(\overline{d}B+A\wedge B)(\overline{U},\overline{V})\textbf{G}(\overline{W},\overline{X},\overline{Y},\overline{Z})
  }&=&0.
\end{eqnarray*}
Therefore, $\overline{d}B+A\wedge B$ vanishes identically.

Finally,  the proof of ({\bf c}) follows from ({\bf b}) and (\ref{eq.16}).
\end{proof}

\begin{prop} \label{prop.3a} Let $(M,L)$ be a  horizontally integrable generalized 2-recurrent Finsler manifold of dimension $n\geq3$ with
recurrence forms $\alpha, \mu$ and non-zero  constant scalar curvature $r$, then we have:
  \begin{description}
  \item[(a)]  $\alpha$ and $\mu$ are symmetric scalar 2-forms.
  \item[(b)] $R(\overline{X},\overline{Y}){\bf R}=0$.
 \end{description}
     \end{prop}

%%%%%%%%%%%%%%%%%%%%%%%%%%%%%%%%%%%%%%%%%%%%%%%%%%%%%%%%%%%%% Section 6 %%%%%%%%%%%%%%%%%%%%%%%%%%%%%%%%%%%%%%%%%%%%%%%%%%%%%%%%%%%%%%%%%

\Section{ Generalized Concircular recurrence}

In this section, we study a new type of Finsler recurrence, namely the generalized concircular recurrence, which generalizes the concircular recurrence investigated in \cite{F1} by the present authors.

\begin{thm}\label{before last}
Let $(M,L)$ be a generalized concircularly recurrent Finsler manifold of dimension $n\geq3$ with
recurrence forms $A, B$ and scalar curvature $r$, then
\begin{description}
  \item[(a)] $(M,L)$ is $G(2CF_{n}) $ provided that $\stackrel{h}{\nabla}A +A\otimes A\neq0$ and $\stackrel{h}{\nabla}B +A\otimes B\neq0$,
  \item[(b)] $(M,L)$ is $GF_{n} $ \emph{(}resp. $GRF_{n}$\emph{)} provided that $B-\frac{rA}{n(n-1)}+\frac{\stackrel{h}{\nabla}{r}}{n(n-1)}\neq0$.
  \end{description}
\end{thm}

\begin{proof} ~\par
\vspace{4pt}
 \noindent\textbf{(a)}  Let $(M,L)$ be generalized concircularly recurrent.
 Then, $\stackrel{h}{\nabla} {C}= A\otimes {C}+B\otimes {G},$  with $ A\neq 0 \neq B$. Consequently,
      \begin{eqnarray}
      \stackrel{h}{\nabla} \stackrel{h}{\nabla} {C}&{=}& (\stackrel{h}{\nabla}A\otimes {C}+A\otimes \stackrel{h}{\nabla}{C})
       +\stackrel{h}{\nabla}B\otimes {G}, \nonumber \\
        &=& (\stackrel{h}{\nabla}A\otimes {C}
         +A\otimes A\otimes {C}+A\otimes B\otimes {G})
       +\stackrel{h}{\nabla}B\otimes {G}   \nonumber \\
       &{=}& (\stackrel{h}{\nabla}A
         +A\otimes A)\otimes {C}+(\stackrel{h}{\nabla}B+A\otimes B)\otimes {G}   \nonumber \\
          &{=}& \alpha\otimes {C}+\mu\otimes {G}   \nonumber .
      \end{eqnarray}
where $\alpha:=\stackrel{h}{\nabla}A +A\otimes A$
 and $\mu:=\stackrel{h}{\nabla}B+A\otimes B$. If $\alpha$ and $\mu$ are none-zero, then  $(M,L)$ is $G(2CF_{n})$.

\vspace{4pt}

\noindent\textbf{(b)}  As $\stackrel{h}{\nabla} {C}= A\otimes {C}+B\otimes {G},$  with $ A\neq 0 \neq B$, then
      \begin{eqnarray}
      \stackrel{h}{\nabla} ({R}-\frac{r}{n(n-1)}\, {G})&{=}&
        A\otimes ({R}-\frac{r}{n(n-1)}\, G)+B\otimes {G}  \nonumber \\
          \stackrel{h}{\nabla}{R}-\frac{\stackrel{h}{\nabla}{r}}{n(n-1)}\otimes {G}&{=}&
        A\otimes ({R}-\frac{r}{n(n-1)}\, G)+B\otimes {G},  \,\, \text{since} \stackrel{h}{\nabla}G=0 \nonumber \\
         \stackrel{h}{\nabla}{R}&{=}&
        A\otimes {R}+(B-\frac{rA}{n(n-1)}+\frac{\stackrel{h}{\nabla}{r}}{n(n-1)})\otimes {G}   \label{eq.2}\\
        \stackrel{h}{\nabla}{R}&{=}&
        A\otimes {R}+B_{1}\otimes {G}\nonumber ,
      \end{eqnarray}
where   $B_{1}:=B-\frac{rA}{n(n-1)}+\frac{\stackrel{h}{\nabla}{r}}{n(n-1)}$. Since $B_{1}\neq0$, then $(M,L)$ is $GF_{n}$.

         \vspace{2pt}

 On the other hand, from (\ref{eq.2}), we obtain
      \begin{eqnarray}
        \stackrel{h}{\nabla}\text{Ric}&{=}&A\otimes \text{Ric}+\frac{1}{n}(n(n-1)B-rA+\stackrel{h}{\nabla}r)\otimes {g} \label{new}\\
        &{=}& A\otimes \text{Ric}+B_{2}\otimes {g}  \nonumber ,
      \end{eqnarray}
 where $B_2=(n-1)B_1$. This completes the proof.
\end{proof}

\begin{lem}\label{thm.1c} Let $(M,L)$ be a horizontally integrable generalized concircularly recurrent Finsler manifold with recurrence forms $A$ and $B$.
The scalar curvature $r$ of $(M,L)$ is horizontally parallel if and only if\, $\, 2rA=2n\,A\circ \text{\em{Ric}}_{o}-n(n-1)(n-2)B$,
where ${\text{\em{Ric}}}_o$ is defined by $g({\text{\em{Ric}}}_{o}\overline{X},\overline{Y}) := \text{\em{Ric}}(\overline{X},\overline{Y})$.
\end{lem}

\begin{proof}
By (\ref{eq.2}) and Lemma \ref{lem.1}{\bf(c)}, we obtain
\begin{eqnarray*}
       && A(\overline{W})R(\overline{X},\overline{Y})\overline{Z}+A(\overline{X})R(\overline{Y},\overline{W})\overline{Z}
        +A(\overline{Y})R(\overline{W},\overline{X})\overline{Z}\\
        &&+\frac{1}{n(n-1)}\{n(n-1)B(\overline{W})- rA(\overline{W})+\stackrel{h}{\nabla}{r}(\overline{W})\}
         \{g(\overline{X},\overline{Z})\overline{Y}-g(\overline{Y},\overline{Z})\overline{X}\}\\
        &&+\frac{1}{n(n-1)}\{n(n-1)B(\overline{X})- rA(\overline{X})+\stackrel{h}{\nabla}{r}(\overline{X})\}
         \{g(\overline{Y},\overline{Z})\overline{W}-g(\overline{W},\overline{Z})\overline{Y}\}\\
        &&+\frac{1}{n(n-1)}\{n(n-1)B(\overline{Y})- rA(\overline{Y})+\stackrel{h}{\nabla}{r}(\overline{Y})\}
         \{g(\overline{W},\overline{Z})\overline{X}-g(\overline{X},\overline{Z})\overline{W}\}=0.
\end{eqnarray*}
Contracting the above relation with respect to ${\overline{Y}}$, given that $g(\overline{X},\overline{\sigma}):=A(\overline{X})$, we get
\begin{eqnarray*}
       && A(\overline{W}){\text{{Ric}}}(\overline{X},\overline{Z})-A(\overline{X}){\text{{Ric}}}(\overline{W},\overline{Z})
        +\textbf{R}(\overline{W},\overline{X},\overline{Z}, \overline{\sigma})\\
        &&+\frac{1}{n}\{n(n-1)B(\overline{W})- rA(\overline{W})+\stackrel{h}{\nabla}{r}(\overline{W})\}
         g(\overline{X},\overline{Z})\\
        &&-\frac{1}{n}\{n(n-1)B(\overline{X})- rA(\overline{X})+\stackrel{h}{\nabla}{r}(\overline{X})\}
        g(\overline{Z},\overline{W})\\
        &&+\frac{1}{n(n-1)}\{n(n-1)B(\overline{X})- rA(\overline{X})+\stackrel{h}{\nabla}{r}(\overline{X})\}
         g(\overline{W},\overline{Z})\\
        &&-\frac{1}{n(n-1)}\{n(n-1)B(\overline{W})- rA(\overline{W})+\stackrel{h}{\nabla}{r}(\overline{W})\}
         g(\overline{X},\overline{Z})=0.
\end{eqnarray*}
This Relation reduces to
\begin{eqnarray*}
       && A(\overline{W})\text{Ric}(\overline{X},\overline{Z})-A(\overline{X})\text{Ric}(\overline{W},\overline{Z})
        +\textbf{R}(\overline{W},\overline{X},\overline{Z}, \overline{\sigma})\\
        &&+\frac{n-2}{n(n-1)}\{n(n-1)B(\overline{W})- rA(\overline{W})+\stackrel{h}{\nabla}{r}(\overline{W})\}
         g(\overline{X},\overline{Z})\\
        &&-\frac{n-2}{n(n-1)}\{n(n-1)B(\overline{X})- rA(\overline{X})+\stackrel{h}{\nabla}{r}(\overline{X})\}
        g(\overline{Z},\overline{W})=0.
\end{eqnarray*}
Contracting the above relation with respect to $\overline{X}$ and $\overline{Z}$, we obtain
$$2rA(\overline{W})-2nA(\text{Ric}_{o}\overline{W})+n(n-1)(n-2)B(\overline{W})+(n-2)\stackrel{h}{\nabla}{r}(\overline{W})=0.$$
Hence, the result follows.
\end{proof}

\begin{thm}\label{last}
Let $(M,L)$ be a horizontally integrable generalized concircularly recurrent Finsler manifold with recurrence forms $A$ and $B$.
If the scalar curvature $r$ of $(M,L)$ is constant, then $(M,L)$ is a $GRF_{n}$.
\end{thm}

\begin{proof}
If the scalar curvature $r$ of $(M,L)$ is constant, then $\stackrel{h}{\nabla}r=0$. Hence, in view of Lemma \ref{thm.1c}, we get
\begin{equation} \label{new2}
    rA=nA\circ {\text{{Ric}}}_{o}-\frac{n(n-1)(n-2)}{2}B.
\end{equation}
 From which, together (\ref{new}), we obtain
\begin{eqnarray}
               \stackrel{h}{\nabla}{\text{Ric}}&{=}&
        A\otimes {\text{Ric}}+\left(\frac{n(n-1)}{2}B-A\circ \text{Ric}_{o}\right)\otimes {g} \nonumber \\
        &=&A\otimes {\text{Ric}}+D \otimes g \label{feq} ,
 \end{eqnarray}
where $D:=\frac{n(n-1)}{2}B-A\circ \text{Ric}_{o}$.
\par
Now, we show that $D\neq0$. Assume the contrary,  then
$$A\circ \text{Ric}_{o}=\frac{n(n-1)}{2}B.$$
 Substituting into (\ref{new2}), we obtain
\begin{equation}\label{tt}
  rA=n(n-1)B.
\end{equation}
From which together with (\ref{eq.2}), noting that the scaler curvature $r$ is constant, we get $rA=0$.
Hence, again by (\ref{tt}), $B=0$. This is a contradiction.\\
Therefore, by (\ref{feq}), $(M,L)$ is  $GRF_{n}$ (as $D\neq0$).
\end{proof}

\begin{rem}
Both Theorem \ref{before last}(b) and Theorem \ref{last} state roughly that, under certain conditions, a $GCF_n$ manifold is a $GRF_n$ manifold. The difference between the two results is that in Theorem \ref{before last}(b) the condition is posed on the recurrence forms $A$ and $B$, whereas in Theorem \ref{last} the condition is posed on the geometric structure of the underlying manifold ($r$ is constant, $\widehat{R}=0$).
\end{rem}

\bigskip
\noindent\textbf{Concluding Remarks}

\vspace{6pt}
Three classes of recurrence in Finsler geometry  are introduced and investigated. The interrelationships among these classes of recurrence are studied.
The following diagram presents concisely the most important results of the paper, where an arrow means "if ... then".
Here are some comments on this table:
\begin{itemize}
  \item Continuous arrows represent results (theorems) of the paper. Dashed arrows represent examples of results that can be deduced from continuous arrows.
  \item Conditions posed on the recurrence forms are not written in the diagram. The written conditions are those posed on the geometric structure of the underlying manifold.
  \item One can deduce the following result from the table:
  $$F_n\stackrel{r\neq0, \widehat{R}=0}\Longleftrightarrow CF_n$$
  This is one of the main result of \cite{F1}.
  \item Among other new important results that can be deduced from the table, we set:
  $$GF_n\Longleftrightarrow CF_n$$
  $$F_n\stackrel{r\neq0, \widehat{R}=0}\Longleftrightarrow GF_n$$
  $$GCF_n\stackrel{r\neq0}\Longrightarrow CF_n$$
\end{itemize}

%\begin{landscape}
\begin{center}\vspace{-0.5 cm}
\includegraphics[width = 17 cm ]{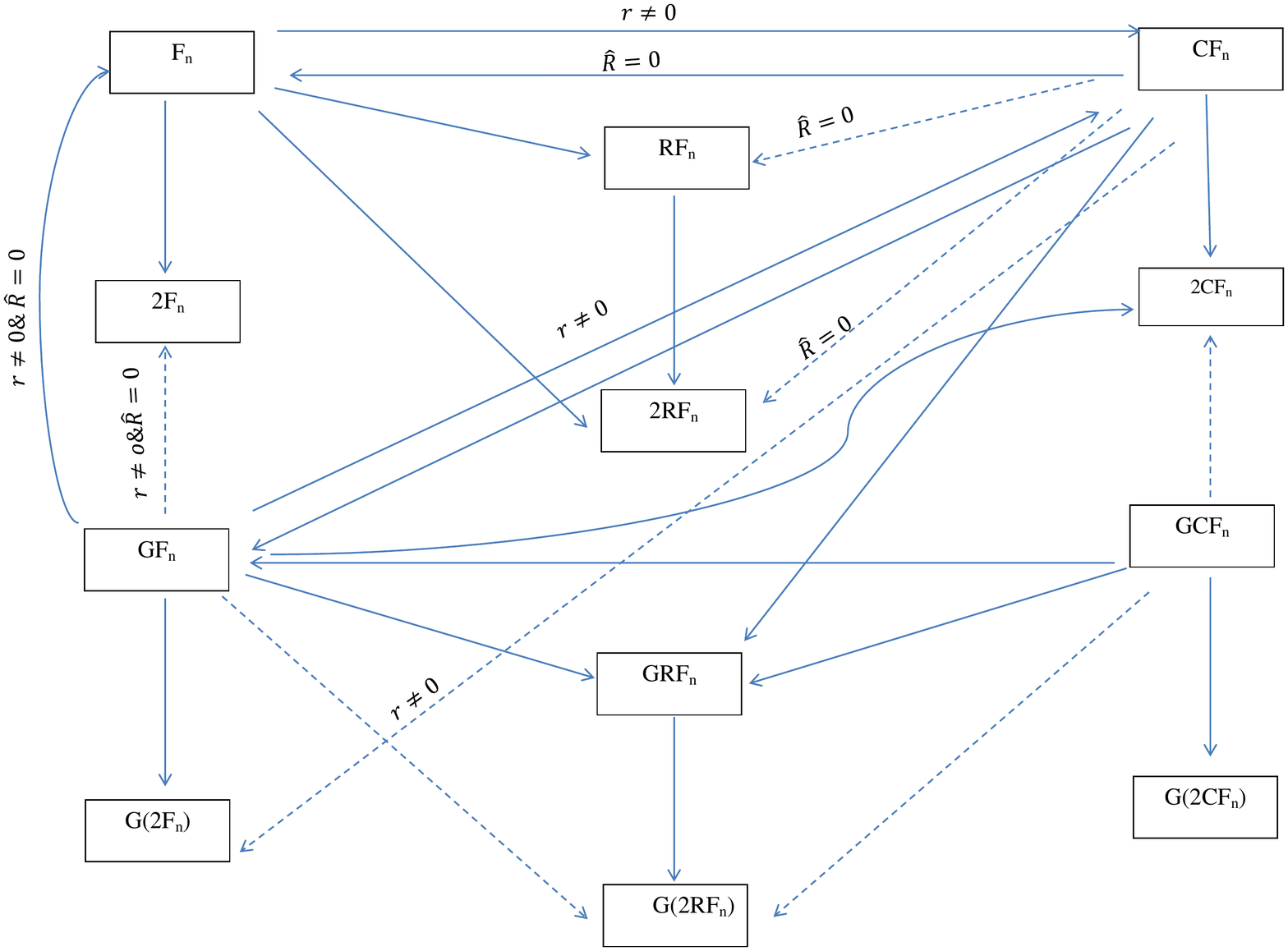}
\end{center}
\vspace{-0.9 cm}
\begin{center}\small{Figure 1. Relationships among different types of recurrent Finsler manifolds}
\end{center}
%\end{landscape}

%%%%%%%%%%%%%%%%%%%%%%%%%%%%%%%%%%%%%%%%%%%%%%%%%%%% Refrences %%%%%%%%%%%%%%%%%%%%%%%%%%%%%%%%%%%%%%%%%%%%%%%%%

\providecommand{\bysame}{\leavevmode\hbox
to3em{\hrulefill}\thinspace}
\providecommand{\MR}{\relax\ifhmode\unskip\space\fi MR }
\providecommand{\MRhref}[2]{%
  \href{http://www.ams.org/mathscinet-getitem?mr=#1}{#2}
} \providecommand{\href}[2]{#2}


\begin{thebibliography}{21}

\bibitem{r58}
H.~Akbar-Zadeh, \emph{Initiation to global Finsler geometry},
Elsevier, 2006.

%\bibitem{R1}
%U. C. De and  N. Guha, \emph{On generalized recurrent manifold}, J. National Academy
%of Math. India, \textbf{9} (1991), 85-92.

\bibitem{R3}
U. C. De, N. Guha and D. Kamilya, \emph{On generalized Ricci-recurrent manifolds},
Tensor, N. S., \textbf{56} (1995), 312-317.

\bibitem{R4}
Y. B. Maralabhavi and M. Rathnamma, \emph{On generalized recurrent manifold},
Indian J. Pure Appl. Math., \textbf{30} (1999), 1167-1171.

\bibitem{F3}
M.~Matsumoto, \emph{On $h$-isotropic and $\textsc{C}^{h}$-recurrent
Finsler spaces}, J. Math. Kyoto Univ., \textbf{11} (1971), 1-9.

\bibitem{F2}
R. S. Mishra and H. D. Pande, \emph{Recurrent Finsler spaces}, J. Ind. Math. Soc. \textbf{32} (1968) 17-22.

\bibitem{R2}
E. M. Patterson, \emph{Some theorems on Ricci recurrent spaces}, J. London Math. Soc.,
\textbf{27} (1952), 287-295.

%\bibitem{r44a}
%A.~A. Tamim, \emph{Special Finsler manifolds}, J. Egypt. Math. Soc.,
%\textbf{10}, \textbf{2} (2002), 149-177.

\bibitem{R5}
H. Singh and Q. Khan, \emph{On generalized recurrent Riemannian manifolds}, Publ.
Math. Debrecen, \textbf{56} (2000), 87-95.

\bibitem{R1}
A. G. Walker, \emph{On Ruses's spaces of recurrent curvature}, Proc.
London Math. Soc.,  \textbf{52} (1950), 36-64.

\bibitem{F1}
Nabil~L. Youssef and A.~Soleiman, \emph{On concircularly recurrent  Finsler manifolds},
   Balkan J. Geom. Appl., \textbf{18}, \textbf{1} (2013), 101-113.
    arXiv: 0704.0053 [math. DG].


\bibitem{r86}
Nabil~L. Youssef, S.~H. Abed and A.~Soleiman, \emph{A global
approach to the theory of special Finsler manifolds},
   J. Math. Kyoto Univ., \textbf{48}, \textbf{4} (2008), 857-893.
    arXiv: 0704.0053 [math. DG].

\bibitem{r94}
\bysame, \emph{A global approach to  the theory  of connections in
Finsler geometry}, \linebreak Tensor, N. S., \textbf{71},\textbf{ 3} (2009),
187-208. arXiv: 0801.3220 [math.DG].

\bibitem{r96}
\bysame, \emph{Geometric objects associated with the fundumental
connections in Finsler geometry},  J. Egypt. Math. Soc.,
\textbf{18},  \textbf{1} (2010), 67-90. arXiv: 0805.2489 [math.DG].

\end{thebibliography}
\end{document}